\documentclass[11pt, draft]{amsart}
\usepackage[margin=3.7cm]{geometry}
\usepackage{amssymb, amstext, amscd, amsmath, amsthm}
\usepackage{mathtools, color, paralist, dsfont, rotating}
\usepackage{verbatim}
\usepackage{enumerate}
\usepackage{soul}

\usepackage{tikz-cd}
\usepackage{tikz}
\usepackage[all]{xy}
\numberwithin{equation}{section}

\makeatletter
\def\bbordermatrix#1{\begingroup \m@th
  \@tempdima 4.75\p@
  \setbox\z@\vbox{%
    \def\cr{\crcr\noalign{\kern2\p@\global\let\cr\endline}}%
    \ialign{$##$\hfil\kern2\p@\kern\@tempdima&\thinspace\hfil$##$\hfil
      &&\quad\hfil$##$\hfil\crcr
      \omit\strut\hfil\crcr\noalign{\kern-\baselineskip}%
      #1\crcr\omit\strut\cr}}%
  \setbox\tw@\vbox{\unvcopy\z@\global\setbox\@ne\lastbox}%
  \setbox\tw@\hbox{\unhbox\@ne\unskip\global\setbox\@ne\lastbox}%
  \setbox\tw@\hbox{$\kern\wd\@ne\kern-\@tempdima\left[\kern-\wd\@ne
    \global\setbox\@ne\vbox{\box\@ne\kern2\p@}%
    \vcenter{\kern-\ht\@ne\unvbox\z@\kern-\baselineskip}\,\right]$}%
  \null\;\vbox{\kern\ht\@ne\box\tw@}\endgroup}
\makeatother

\makeatletter
\def\@cite#1#2{{\m@th\upshape\bfseries%
[{#1\if@tempswa{\m@th\upshape\mdseries, #2}\fi}]}}
\makeatother
\theoremstyle{plain}
\newtheorem{theorem}{Theorem}[section]
\newtheorem{corollary}[theorem]{Corollary}

\newtheorem{lemma}[theorem]{Lemma}
\theoremstyle{definition}
\newtheorem{definition}[theorem]{Definition}
\newtheorem{example}[theorem]{Example}

\newtheorem{remark}[theorem]{Remark}

\theoremstyle{remark}
\mathtoolsset{centercolon}

\renewcommand{\H}{{\mathcal{H}}}

\renewcommand{\L}{{\mathcal{L}}}
\newcommand{\M}{{\mathcal{M}}}

\newcommand{\T}{{\mathcal{T}}}

\newcommand{\bN}{\mathbb{N}}

\newcommand{\bZ}{\mathbb{Z}}
\newcommand{\bR}{\mathbb{R}}

\renewcommand{\bZ}{\mathbb{Z}}

\newcommand{\AND}{\text{ and }}

\newcommand{\FORAL}{\text{ for all }}

\newcommand{\Alg}{\operatorname{Alg}}

\newcommand{\id}{{\operatorname{id}}}

\newcommand{\Ind}{{\operatorname{Ind}}}

\newcommand{\Gr}{\mathrm{Gr}}

\newcommand{\supp}{\operatorname{supp}}
\newcommand{\Arv}{\operatorname{Arv}}
\newcommand{\LL}{\operatorname{LL}}

\newcommand{\p}[2]{P_{#1}^{(#2)}}
\newcommand{\q}[2]{Q_{#1}^{(#2)}}
\newcommand{\s}[2]{S_{#1}^{(#2)}}
\renewcommand{\r}[2]{r_{#1}^{(#2)}}
\renewcommand{\a}[2]{\alpha_{#1}^{(#2)}}

\begin{document}
%%%%%%%%%%%%%%%%%%%%%%%%%%%%%%%%

%%%%%%%%%%%%%%%%%%%%%%%%%%%%%%%%
\title[Doob equivalence and NC peaking for Markov chains]{Doob equivalence and non-commutative peaking for Markov chains}

\author[X. Chen]{Xinxin Chen}
\address{Department of Statistics\\ University of Chicago\\ Chicago\\ IL \\ USA}
\email{xinxinc@galton.uchicago.edu}

\author[A. Dor-On]{Adam Dor-On}
\address{Department of Mathematical Sciences\\ University of Copenhagen\\ Copenhagen\\ Denmark}
\email{adoron@math.ku.dk}

\author[L. Hui]{Langwen Hui}
\address{Department of Mathematics\\ University of Illinois at Urbana-Champaign\\ Urbana\\ IL \\ USA}
\email{langwen3@illinois.edu}

\author[C. Linden]{Christopher Linden}
\address{Department of Mathematics\\ University of Illinois at Urbana-Champaign\\ Urbana\\ IL \\ USA}
\email{clinden2@illinois.edu}

\author[Y. Zhang]{Yifan Zhang}
\address{Department of Mathematics\\ University of Illinois at Urbana-Champaign\\ Urbana\\ IL \\ USA}
\email{yifan8@illinois.edu}


\thanks{This work was done as part of the research project ``Problems on Markov chains arising from operator algebras" at the Illinois Geometry Lab in Spring 2019. The first, third and fifth authors participated as undergraduate scholars, the fourth author served as graduate student team leader, and the second author together with Florin Boca served as faculty mentors. The project was supported by an NSF grant DMS-1449269. The second author was supported by an NSF grant DMS-1900916 and by the European Union's Horizon 2020 Marie Sklodowska-Curie grant No 839412.}

\maketitle

%%%%%%%%%%%%%%%%%%%%%%%%%%%%%%%%
\begin{abstract}
In this paper we show how questions about operator algebras constructed from stochastic matrices motivate new results in the study of harmonic functions on Markov chains. More precisely, we characterize coincidence of conditional probabilities in terms of (generalized) Doob transforms, which then leads to a stronger classification result for the associated operator algebras in terms of spectral radius and strong Liouville property. Furthermore, we characterize the non-commutative peak points of the associated operator algebra in a way that allows one to determine them from inspecting the matrix. This leads to a concrete analogue of the maximum modulus principle for computing the norm of operators in the ampliated operator algebras.
\end{abstract}

%%%%%%%%%%%%%%%%%%%%%%%%%%%%%%%%
\section{Introduction}
%%%%%%%%%%%%%%%%%%%%%%%%%%%%%%%%

The theory of Markov chains has applications in diverse areas of research such as group theory, dynamical systems, electrical networks and information theory. For the basic theory of Probability, Markov chains, random walks and their applications we refer the reader to \cite{Dur10, Fel68, Woe00}. These days, connections with operator algebras seem to manifest mostly in quantum information theory, where Markov chains are generalized to quantum channels. For the basic theory of operator on Hilbert space and their algebras we refer the reader to \cite{Arv76, Dav96, Dix77}. In this paper we resolve problems related to Markov chains motivated from studying operator algebras associated to stochastic matrices as in \cite{DOM14, DOM16}. 

\begin{definition}
 Let $\Omega$ be a countable set. A \emph{stochastic matrix} is a function $P :\Omega \times \Omega \rightarrow [0,1]$ such that for all $i\in \Omega$ we have $\sum_{j \in \Omega} P_{ij} = 1$. We let $\Gr(P)$ be the directed graph on $\Omega$ with directed edges $(i,j)$ when $P_{ij}>0$. We say that $P$ is \emph{irreducible} if $\Gr(P)$ is a strongly connected directed graph.
\end{definition}

The set $\Omega$ is called the state space of $P$. We denote by $P_{ij}^{(n)}:= (P^n)_{ij}$ the $ij$-th entry of $P^n$ for $n \in \bN$, and  we agree that $P^{(0)} = I$ will always mean the identity matrix. Next we define some analytic properties of $P$.

\begin{definition}
Let $P$ be a stochastic matrix over $\Omega$. 
\begin{enumerate}
\item
A state $i \in \Omega$ is \emph{recurrent} if the expected number of returns $\sum_{n\in \bN} P_{ii}^{(n)}$ is infinite. Otherwise we say that $i$ is transient. We say that $P$ is recurrent / transient if all of its states are recurrent / transient respectively.

\item
The \emph{spectral radius} of a state $i\in \Omega$ is $\limsup \sqrt[n]{P_{ii}^{(n)}}$. If every state has the same spectral radius, we will denote this $\rho(P)$ and call it the spectral radius of $P$. If $\rho(P)=1$, we say that $P$ is \emph{amenable}.
\end{enumerate}
\end{definition}

When $P$ is irreducible, every state has the same spectral radius, and a state $i \in \Omega$ is recurrent if and only if $P$ is recurrent. Clearly any recurrent stochastic matrix is amenable. The terminology of ``amenable" stochastic matrix is justified by Kesten's amenability criterion \cite{Kes59} (see also the discussion after Definition \ref{d:random-walk}).

Let $\{e_i\}_{i=1}^d$ be the standard basis for $\mathbb{R}^d$. A famous result of Polya from 1921 states that the simple random walk $P$ on $\mathbb{Z}^d$ given by $P_{x,x+e_i} = P_{x,x-e_i} = \frac{1}{2d}$ is recurrent if and only if $d\leq 2$. On the other hand, since $\mathbb{Z}^d$ is an amenable group, we see from Kesten's amenability criterion that $P$ is amenable as a stochastic matrix.

In \cite{DOM14} non-self-adjoint operator algebras $\T_+(P)$ associated to stochastic matrices were studied. A combination of \cite[Theorem 3.8]{DOM14} and \cite[Theorem 7.27]{DOM14} then shows that two stochastic matrices $P$ and $Q$ have isometrically isomorphic tensor algebras if and only if they have the same conditional probabilities as in item (ii) of the definition below. In this way this equivalence relation arises naturally from the solution of the isometric isomorphism problem for $\T_+(P)$ from \cite{DOM14}.

\begin{definition}
Let $P$ and $Q$ be stochastic matrices over $\Omega_P$ and $\Omega_Q$ respectively. We say that $P$ and $Q$ are 
\begin{enumerate}
\item
\emph{conjugate}, and denote this by $P \cong Q$, if there is a bijection $\sigma : \Omega_P \rightarrow \Omega_Q$ such that $P_{ij} = Q_{\sigma(i)\sigma(j)}$ for every $(i,j) \in \Gr(P)$.

\item
\emph{Doob equivalent}, and denote this by $P \sim_d Q$, if there exists a bijection $\sigma : \Omega_P \rightarrow \Omega_Q$ which is a graph isomorphism between $\Gr(P)$ and $\Gr(Q)$ such that for all $n,m \in \bN$ and for any $(i,k) \in \Gr(P^{n+m})$ we have
$$
\frac{P_{ij}^{(n)}P_{jk}^{(m)}}{P_{ik}^{(n+m)}} = \frac{Q_{\sigma(i)\sigma(j)}^{(n)}Q_{\sigma(j)\sigma(k)}^{(m)}}{Q_{\sigma(i)\sigma(k)}^{(n+m)}}.
$$
\end{enumerate}
\end{definition}

Our first main result of this paper is the characterization of Doob equivalence in terms of (generalized) Doob transforms for irreducible stochastic matrices. More precisely, in Theorem \ref{t:char-cond} we show that two irreducible stochastic matrices $P$ and $Q$ are Doob equivalent via $\sigma$ if and only if there exists a \emph{positive} function $h : \Omega_P \rightarrow \bR_+$ and eigenvalue $\lambda >0$ such that $Ph = \lambda h$, and $Q_{\sigma(i)\sigma(j)} = P^{(h,\lambda)}_{ij}:= \lambda^{-1} \frac{h(j)}{h(i)}P_{ij}$. This operation of applying an eigenpair $(h,\lambda)$ of $P$ to obtain a new stochastic matrix $P^{(h,\lambda)}$ as above is called (generalized) \emph{Doob transform} or \emph{h-transform} in the literature. 

When the eigenvalue $\lambda$ above is equal to $1$ we call $h$ \emph{harmonic}, and Doob transforms by $(h,1)$ are used to condition the transition probabilities of a Markov chains, as well as to study Martin boundaries of random walks. Beyond their intrinsic interest as discrete analogues from complex analysis, harmonic functions of Markov chains are intimately related with the behavior of a Markov chain $\{X_n\}$ as $n \rightarrow \infty$. For more on the theory of harmonic functions and the Martin boundary of Markov chains we refer the reader to \cite{KKS76}.

One of our motivating questions comes from \cite[Theorem 3.11]{DOM14} where it is shown that if $P$ and $Q$ are irreducible, recurrent and $P \sim_d Q$ with graph isomorphism $\sigma$, then $P$ and $Q$ are conjugate via $\sigma$. We were interested in finding optimal conditions on $P$ and $Q$ that guarantee that $P \sim_d Q$ implies $P \cong Q$. When this occurs, the isometric isomorphism of $\T_+(P)$ and $\T_+(Q)$ implies conjugacy of $P$ and $Q$.

It turns out that harmonic functions naturally come into play in the solution of this problem. Let $P$ be stochastic over $\Omega$. We say that that $P$ is \emph{strong Liouville} if all positive harmonic functions are constant. Strong Liouville property for an irreducible stochastic matrix means that the associated Markov chain $\{X_n\}$ converges in probability to a unique point as $n \rightarrow \infty$, regardless of initial distribution. Strong Liouville property also manifests naturally in other areas. For instance, in Riemannian geometry a famous result of Yau \cite{Yau75} shows that Riemannian manifolds with non-negative Ricci curvature have the strong Liouville property. 

A result of Derman \cite{Der54} shows that recurrent stochastic matrices are strong Liouville, so we ask if recurrence can be weakened to amenability or to strong Liouville property in \cite[Theorem 3.11]{DOM14}. Our characterization of Doob equivalence solves this problem and allows us to show (see Corollary \ref{c:optimal-classification}) that if $P$ and $Q$ are irreducible stochastic matrices with $\rho(P) = \rho(Q)$ such that either $P$ or $Q$ are strong Liouville, then the isometric isomorphism of $\T_+(P)$ and $\T_+(Q)$ implies conjugacy of $P$ and $Q$. This result is optimal in the sense that Example \ref{ex:amenable-non-SL} shows we cannot weaken recurrence in \cite[Theorem 3.11]{DOM14} to amenability without assuming strong Liouville property for one of the matrices.

The operator algebra $\T_+(P)$ comes equipped with the operator norm induced by its embedding into bounded operators on a specific Hilbert space $\H_P$ defined in Section \ref{s:prelim}. Once this is done, for $\ell\geq 1$ the identification $M_{\ell}(B(\H_P)) \cong B(\oplus_{n=1}^{\ell}\H_P)$ gives rise to a natural operator norm on $M_{\ell}(\T_+(P))$ as well via the embedding $M_{\ell}(\T_+(P)) \subseteq M_{\ell}(B(\H_P))$. From the definition of $\T_+(P)$ in Section \ref{s:prelim} it follows that $\H_P$ has minimal reducing subspaces $\H_{P,k}$ for $\T_+(P)$, each of which is associated to a state $k\in \Omega$.

\begin{definition}
Let $P$ be a \emph{finite} irreducible stochastic matrix over $\Omega$. A state $k \in \Omega$ is called \emph{completely peaking} for $\T_+(P)$ if there is an operator $T = [T_{pq}] \in M_{\ell}(\T_+(P))$ for some $\ell \geq 1$ such that
$$
\| [T_{pq}|_{\H_{P,k}}] \| > \max_{s \neq k} \| [T_{pq}|_{\H_{P,s}}] \|.
$$
Denote all completely peaking states by $\Omega_b$.
\end{definition}

When we associate to $k\in \Omega$ a representation $\pi_k :C^*(\T_+(P)) \rightarrow B(\H_{P,k})$ given by $\pi_k(T) = T|_{\H_{p,k}}$, we get that $k$ is completely peaking if and only if the representation $\pi_k$ is strongly peaking in the sense of \cite[Definition 7.1]{Arv11}. The study of peaking representations originates from Arveson's pioneering work on non-commutative analogues of Shilov and Choquet boundaries for operator algebras (see for instance \cite{Arv69, Arv72, Arv08, Arv11}).

In \cite{DOM16}, completely peaking states of $P$ were computed under the assumption of multiple arrival (see \cite[Corollary 3.14]{DOM16}). Based on this, the $C^*$-envelope of $\T_+(P)$ is computed and classified (see \cite[Section 3]{DOM16} and \cite[Theorem 5.5 \& Theorem 5.6]{DOM16}). In order to complete the picture, in \cite[Remark 3.12]{DOM16} it was asked if the assumption of multiple arrival can be dropped in \cite[Corollary 3.14]{DOM16}. 

Let $P$ be a finite irreducible stochastic matrix over $\Omega$. We say that $k \in \Omega$ is \emph{escorted} if for all $s\in \Omega$ with $P_{ks} > 0$ there exists $k' \neq k$ in $\Omega$ such that $P_{k's} > 0$. We strengthen \cite[Proposition 3.11 \& Corollary 3.14]{DOM16} and show in Theorem \ref{t:esc-bdry} that $k \in \Omega_b$ if and only if $k$ is escorted. This completes the picture in \cite{DOM16}, and answers the question in \cite[Remark 3.12]{DOM16} negatively through Example \ref{e:nmp}.

As an application (see Corollary \ref{c:max-mod}), we get a formula which reduces the computation of the norm of elements in $M_{\ell}(\T_+(P))$ to norms of the restrictions to the reducing subspaces $\H_{P,k}$ associated to escorted states $k$. This is a non-commutative analogue of the maximum modulus principle for elements in $M_{\ell}(\T_+(P))$.

%%%%%%%%%%%%%%%%%%%%%%%%%%%%%%%%%%%%%%%%%%%%

This paper is divided into four sections including this introductory section. In Section \ref{s:prelim} we review some of the theory of random walks, the construction of the tensor algebra of a stochastic matrix and its completely peaking states. In Section \ref{S:Doob-equiv} we prove our first main result on the characterization of Doob equivalence. As a consequence, we obtain an strong rigidity result (Corollary \ref{c:optimal-classification}) which improves the combination of \cite[Theorem 3.11 \& Theorem 7.27]{DOM14}. Finally, in Section \ref{S:NC-peaking} we prove our second main result on completely peaking states and touch upon some of its applications.
	
%%%%%%%%%%%%%%%%%%%%%%%%%%%%%%%%
\section{Preliminaries} \label{s:prelim}
%%%%%%%%%%%%%%%%%%%%%%%%%%%%%%%%

Here we discuss some of the theory in probability and operator algebras, mostly to do with Markov chains and the construction of their tensor algebras. First we discuss examples of Markov chains arising from countable discrete groups.

\begin{definition} \label{d:random-walk}
Let $G$ be a discrete group, and $\mu$ a probability measure on $G$ such that $\supp (\mu)$ generates $G$ as a semigroup. A \emph{random walk} on $G$ is a stochastic matrix $P$ given by $P_{g,h} = \mu(g^{-1}h)$. We say that a random walk $P$ is \emph{symmetric} on $G$ if $\mu(g) = \mu(g^{-1})$ for any $g\in G$. We will say that a symmetric random walk on $G$ is \emph{simple} if $\mu$ is uniform on its support.
\end{definition}

Kesten's amenability criterion \cite{Kes59} then says that if some symmetric random walk on $G$ is amenable then $G$ is amenable as a group, and conversely that if $G$ is amenable as a group then \emph{all} symmetric random walks on $G$ are amenable. When $P$ is a simple (symmetric) random walk on a group $G$ determined by (uniform) $\mu$, then $S:=\supp(\mu)$ must be a finite generating subset of $G$ such that $S^{-1} = S$, and for every $g \in S$ we have $\mu(g) = \frac{1}{|S|}$.

\begin{definition}
Let $P$ be an irreducible stochastic matrix over $\Omega$. A non-zero function $h:\Omega\rightarrow \mathbb{R}$ is an \emph{eigenfunction} of $P$ if $Ph=\lambda h$ for some $\lambda \in \bR$. We say that a function $h$ is \emph{harmonic} if $Ph = h$. We say that a positive harmonic function $h$ is \emph{minimal} if whenever $g$ is harmonic such that $0 \leq g \leq h$, then there is $C\geq 0$ such that $g = C \cdot h$. Finally, we say that $P$ is \emph{strong Liouville} if all positive harmonic functions are constant.
\end{definition}

Let $P$ be an irreducible stochastic matrix over $\Omega$, and let $0\in \Omega$ be some fixed element. The set $K_0$ of positive harmonic functions $h$ of $P$ such that $h(0)=1$ is a compact convex set under the topology of uniform convergence on finite sets. The extreme points of $K_0$ are then exactly the minimal positive harmonic functions $h$ with $h(0)=1$, and by Krein-Milman theorem we get that their closed convex hull is the set $K_0$ back again. The following theorem of Choquet and Deny \cite{CD60} provides us with a characterization of all minimal positive harmonic functions of random walks on $\mathbb{Z}^d$.

\begin{theorem}
Let $P$ be a random walk on $\mathbb{Z}^d$ determined by a measure $\mu$. Then $h\geq 0$ is a minimal harmonic function with $h(0) = 1$ if and only if $h(x) = e^{\langle \alpha, x \rangle}$ for $\alpha \in \mathbb{R}^d$ such that $\sum_{y \in \mathbb{Z}^d} e^{\langle \alpha, y \rangle} \cdot \mu(y) = 1$.
\end{theorem}

A simple multivariate calculus minimization argument shows that the simple random walk on $\mathbb{Z}^d$ above is strongly Liouville. However, if we take a \emph{biased} random walk on $\bZ$ in the sense that $P_{0,1} \neq P_{0,-1}$, the above shows that $P$ is not strong Liouville.

Next we will define operator algebras associated to stochastic matrices as studied in \cite{DOM14, DOM16}. Let $P$ be a stochastic matrix and define for $k \in \Omega$ the Hilbert space $\H_{P,k}$ to be the closed linear span of the orthonormal basis $\{ \xi_{m,j,k} \}_ {(j,k) \in \Gr(P^m)}$. Let $\H_P$ be the direct sum $\oplus_{k\in \Omega} \H_{P,k}$. Fix $n\geq 0$ and let $\Arv(P)_n$ be the collection of complex matrices $A= [a_{ij}]$ over $\Omega$ such that $a_{ij}=0$ whenever $(i,j)\notin \Gr(P^n)$ and $\sup_{j\in \Omega} \sum_{i \in \Omega} |a_{ij}|^2 < \infty$. For each $A \in \Arv(P)_n$ we define an operator $S^{(n)}_A$ on $\H_P$ by setting
$$
S_A^{(n)}(\xi_{m,j,k}) = \sum_{i\in \Omega} a_{ij}\sqrt{\frac{P_{ij}^{(n)}P_{jk}^{(m)}}{P_{ik}^{(n+m)}}} \xi_{n+m,i,k}.
$$
Then $S_A^{(n)}$ is a bounded operator on $\H_P$. The tensor algebra $\T_+(P)$ associated to the stochastic matrix $P$ is then given by
$$
\T_+(P) := \overline{\Alg} \{ \ S_A^{(n)} \ | \ A \in \Arv(P)_n, \ n \geq 0 \ \}.
$$

Although the definition of the operators $S_A^{(n)}$ above uses the conditional probabilities of $P$ explicitly, it is important to mention that the algebra $\T_+(P)$ arises abstractly from an \emph{Arveson-Stinespring subproduct system} associated to the matrix $P$ (See \cite[Theorem 3.4]{DOM14}), and that extracting these conditional probabilities from the isometric isomorphism class of $\T_+(P)$ is the main thrust of \cite{DOM14}. Subproduct systems originate from the study of quantum Markov processes in the form of $E_0$-semigroups and cp-semigroups \cite{Arv03, BM10} and were studied systematically by Shalit and Solel in \cite{SS09}.

\begin{remark}
Let $P$ be a stochastic matrix over $\Omega$. The definition of $\T_+(P)$ in \cite[Definition 6.1]{DOM14} is slightly different from the one we give here, and this needs some justification. We do this here, but we use some of the theory of $C^* / W^*$-correspondences used in \cite{DOM14, DOM16} to accomplish this. Let $\rho : \ell^{\infty}(\Omega) \rightarrow B(\ell^2(\Omega))$ be the *-representation of $\ell^{\infty}(\Omega)$ as left multiplication $\rho(f)(g)(i) = f(i)g(i)$ for $i\in \Omega$. By \cite[Corollary 2.74]{RW98} we get that $\rho$ induces a faithful *-representation $\Ind(\rho) : \L(\bigoplus_{n = 0}^{\infty}\Arv(P)_n) \rightarrow B \big(\bigoplus_{n = 0}^{\infty}(\Arv(P)_n \otimes_{\rho} \ell^2(\Omega)) \big)$ given by $\Ind(\rho)(T)(\xi \otimes h) = T\xi \otimes h$. Then $\Ind(\rho)$ is an injective *-representation such that the image of $\T_+(P)$ under $\Ind(\rho)$ coincides with the definition of $\T_+(P)$ given above. More precisely, this is realized via a unitary operator from $\H_P$ to $\bigoplus_{n = 0}^{\infty}(\Arv(P)_n \otimes_{\rho} \ell^2(\Omega))$ given by $\xi_{m,j,k} \mapsto E_{jk} \otimes e_k$ for $(j,k) \in \Gr(P^m)$ where $E_{jk} \in \Arv(P)_m$ is a matrix unit and $e_k$ is the characteristic function of $k \in \Omega$.
\end{remark}

\begin{definition}
Let $P$ be an irreducible stochastic matrix. A cyclic decomposition for $P$ is a partition $\Omega = \Omega_0 \sqcup \Omega_1 \sqcup ... \sqcup \Omega_{p-1}$ such that $P_{ij} >0$ and $i\in \Omega_{\ell}$ imply $j \in \Omega_{\ell +1 \mod p}$. The period $p$ of $P$ is the largest possible number of components in a cyclic decomposition for $P$.
\end{definition}

When $P$ is irreducible, it is easy to show that its period must be finite. Suppose now that $P$ is a finite matrix. For the purpose of computing completely peaking states in \cite{DOM16}, we defined a state $i\in \Omega$ as \emph{exclusive} if it comprises its own cyclic component in a cyclic decomposition for $P$ (see also \cite[Definition 3.9 \& Lemma 3.10]{DOM16} for an equivalent definition). We denote by $\Omega_e$ the set of exclusive states. When at least one state is non-exclusive in $P$ (or equivalently when $P$ is not a cycle), it follows from \cite[Section 3]{DOM16} that there is a \emph{unique} smallest non-empty subset $\Omega_b \subseteq \Omega$ such that for any $T = [T_{pq}] \in M_{\ell}(\T_+(P))$ and $1 \leq \ell \in \bN$ we have
$$
\| T \| = \max_{k \in \Omega_b} \| [T_{pq}|_{\H_{P,k}}] \|.
$$
Hence, the norm of any element $M_{\ell}(\T_+(P))$ for any $1 \leq \ell \in \bN$ is retained by ``evaluating" on appropriate restrictions.

For a finite irreducible stochastic matrix $P$ over $\Omega$, we say that $P$ has \emph{multiple arrival} if whenever $s \in \Omega$ is non-exclusive and there is $k\neq s$  such that $P_{ks}^{(n)} >0$, there is $k' \neq k$ such that $P_{k's}^{(n)} >0$. In \cite[Proposition 3.11]{DOM16} we computed $\Omega_b$ when $P$ has multiple arrival and showed that $\Omega_b = \Omega \setminus \Omega_e$. In Section \ref{S:NC-peaking} we will show that this equality may fail without the assumption of multiple arrival.

%%%%%%%%%%%%%%%%%%%%%%%%%%%%%%%%
\section{Doob equivalence} \label{S:Doob-equiv}
%%%%%%%%%%%%%%%%%%%%%%%%%%%%%%%%

In this section we connect Doob equivalence with Doob transforms, allowing us to weaken the assumption of recurrence in \cite[Theorem 3.11]{DOM14}. To simplify our proofs, we will often suppress the isomorphism $\sigma$ between the graphs of $P$ and $Q$, and assume that $\sigma = \id$ is the identity.

\begin{definition}\label{d:doob}
Let $P$ be an irreducible stochastic matrix. Given a positive non-zero eigenfunction $h$ with a (necessarily positive) eigenvector $\lambda$ for $P$, we define the (generalized) \emph{Doob transform} $P^{(h,\lambda)}$ of $P$ by $h$ via
\begin{equation*}
P^{(h,\lambda)}_{ij} = \lambda^{-1} \frac{h(j)}{h(i)} P_{ij} \qquad  \text{ for } i,\,j \in \Omega.
\end{equation*}
\end{definition}

The entries of $P^{(h,\lambda)}$ are non-negative, and $P^{(h,\lambda)}$ is easily seen to be stochastic.

\begin{lemma}\label{l:doob-ratio} 
Let $Q = P^{(h,\lambda)}_{ij}$ and define $\r{ij}{n} := \lambda^{-n} \frac{h(j)}{h(i)}$. Then $\r{ij}{n}\r{jk}{m} = \r{ik}{n+m}$ and $Q^{(n)}_{ij} = \r{ij}{n} P^{(n)}_{ij}$.
\end{lemma}

\begin{proof}
By definition we clearly see that $\r{ij}{n}\r{jk}{m} = \r{ik}{n+m}$. Then, by definition we also have that $Q_{ij} = \r{ij}{1} P_{ij}$ and $Q^{(0)}_{ij} = \r{ij}{0} P^{(0)}_{ij}$. Then towards proof by induction if we assume $Q^{(n)}_{ij} = \r{ij}{n} P^{(n)}_{ij}$, we get that, 
\begin{equation*}
Q^{(n+1)}_{ij} = \sum_k \q{ik}{n}Q_{kj} = \sum_k \r{ik}{n} \r{kj}{1} \p{ik}{n} P_{kj}  = \r{ij}{n+1}P^{(n+1)}_{ij},
\end{equation*}
so that by induction $Q^{(n)}_{ij} = \r{ij}{n} P^{(n)}_{ij}$ for all $n$.
\end{proof}

\begin{corollary}\label{c:doob-radius}
The spectral radius of  $P^{(h,\lambda)}$ is $\rho(P^{(h,\lambda)}) = \lambda^{-1} \rho(P)$.
\end{corollary}

\begin{theorem} \label{t:char-cond}
Let $P$ be an irreducible stochastic matrix. A stochastic matrix $Q$ is Doob equivalent to $P$ if and only if it is conjugate to a Doob transform of $P$.
\end{theorem}

\begin{proof}
From Lemma \ref{l:doob-ratio} it is immediate that $P^{(h,\lambda)}$ is Doob equivalent to $P$. Conversely, suppose $Q$ is conditionally equivalent to $P$. For $i,\, j \in \Omega$ and $n \in \bN$ such that $\p{ij}{n} > 0$, define 
\begin{equation*}
\a{ij}{n} := \frac{\q{ij}{n}}{\p{ij}{n}},
\end{equation*}
 Fix $i_0 \in \Omega$ and $n_0 \in \bN$ such that $\p{i_0i_0}{n_0} > 0$, and define $\lambda := \left(\a{i_0i_0}{n_0}\right)^{-1/n_0} >0$. Let $d_{ij} := \min\{n\,|\,P^{(n)}_{ij} > 0\}$, and set $h(i_0) = 1$ and $h(i) = \lambda^{d_{i_0i}} \a{i_0i}{d_{i_0i}}$. 

We claim that $h$ is a positive eigenfunction of $P$ with eigenvalue $\lambda$ such that $Q=P^{(h,\lambda)}$, but we first establish some properties of $\alpha$ and $\lambda$.

\begin{enumerate}
    \item $\a{ij}{n}\a{jk}{m} = \a{ik}{n+m}$ whenever $(i,j) \in \Gr(P^n), (j,k) \in \Gr(P^m)$,
    \item $\a{ii}{n} = \lambda^{-n}$  whenever $(i,i) \in \Gr(P^n)$,
    \item $\a{ij}{n+m} = \lambda^{-n}\a{ij}{m}$ whenever  $(i,j) \in \Gr(P^n) \cap \Gr(P^{n+m})$.
\end{enumerate}

Claim (i) follows from the definition of Doob equivalence. To prove (ii), let $m_1,\,m_2 \in \bN$ be such that $\p{i_0i}{m_1},\, \p{ii_0}{m_2} > 0$. In view of (i), 
\begin{equation*}
    \left(\a{ii}{n}\right)^{n_0(m_1 + m_2)} = 
    \left(\a{ii_0}{m_1}\a{i_0i}{m_2}\right)^{n n_0} = 
    \left(\a{i_0i_0}{n}\right)^{n_0(m_1 + m_2)},
\end{equation*}
and
\begin{equation*}
\left(\a{i_0i_0}{n}\right)^{n_0} = 
\left(\a{i_0i_0}{n_0}\right)^n = (\lambda^{-n_0})^n = (\lambda^{-n})^{n_0}.
\end{equation*}		
Thus (ii) is proved. To prove (iii) take $k$ such that $(j,i) \in \Gr(P^\ell) > 0$. Applying (i) and (ii),
\begin{equation*}
\alpha_{ij}^{(n+m)} \alpha_{ji}^{(\ell)}= \alpha_{ii}^{(n +m +\ell)} = \lambda^{-(n +m+ \ell)} = \lambda^{-n} \alpha_{ii}^{(m+\ell)} = \lambda^{-n}\alpha_{ij}^{(m)} \alpha_{ji}^{(\ell)},
\end{equation*}
and (iii) is proved. Taking advantage of property (iii), we see that 
$h(i) = \lambda^n\a{i_0i}{n}$ for all $n$ such that $(i_0,i) \in \Gr(P^n)$. As a consequence, for $(i,j) \in \Gr(P)$, and $n$ such that $\p{i_0i}{n} > 0$, we have
\begin{equation*}
  Q_{ij} 
  = \a{ij}{1} P_{ij} 
  = \frac{  \a{i_0j}{n+1}  }{  \a{i_0i}{n}  }P_{ij} 
  = \lambda^{-1} \, \frac{ \lambda^{n+1} \a{i_0j}{n+1}  }
                         { \lambda^{n} \a{i_0i}{n}  } P_{ij} 
  = \lambda^{-1} \frac{h(j)}{h(i)} P_{ij}.
\end{equation*}
Summing both sides over $j$, we see that $(h, \lambda)$ is an eigenpair for $P$. Hence, $Q = P^{(h,\lambda)}$ is a generalized Doob transform of $P$.
\end{proof}

\begin{corollary} \label{c:amenable-strong-liouville-char}
Let $P$ be an irreducible stochastic matrix. $P$ is strong Liouville if and only if for every stochastic matrix $Q$ such that $P \sim_d Q$ and $\rho(P) = \rho(Q)$ we have $Q \cong P$. In particular, if $P$ is an irreducible, amenable strong Liouville stochastic matrix, then for amenable $Q$ we have that $P \sim_d Q$ implies $P \cong Q$.
\end{corollary}

\begin{proof}
Suppose that $P$ is strong Liouville, $P \sim_d Q$ and $\rho(P) = \rho(Q)$. By Theorem \ref{t:char-cond} we know that $Q$ is obtained from eigen-pair $(\lambda, h)$ of $P$. By Corollary \ref{c:doob-radius} we get that $\lambda = 1$, so the eigenfunction $h$ is harmonic and must then be constant. This means that $P = Q$.

Conversely, suppose that $P \sim_d Q$ and $\rho(P) = \rho(Q)$ imply $P = Q$. Let $h$ be a harmonic function for $P$. By Theorem \ref{t:char-cond} we know that $P \sim_d P^{(h,1)}$ and by Corollary \ref{c:doob-radius} we get that $\rho(P) = \rho(P^{(h,1)})$. Hence, by assumption $P= P^{(h,1)}$, so that $h$ must be constant. This means that $P$ is strong Liouville.
\end{proof}

Recurrent matrices are strong Liouville by \cite{Der54}, and are also amenable. Hence, we get \cite[Theorem 3.11]{DOM14} as a special case of Corollary \ref{c:amenable-strong-liouville-char}. On the other hand, there are many examples of strong Liouville transient stochastic matrices for which Corollary \ref{c:amenable-strong-liouville-char} applies but \cite[Theorem 3.11]{DOM14} does not. See for instance \cite{Eri89} where it is shown that a symmetric random walk on a finitely generated group with polynomial growth of bounded order must be strong Liouville.

We provide examples illustrating that the assumptions of strong Liouville and amenability are logically independent. We first recall an example of an amenable matrix which is not strong Liouville, and then an example of a strong Liouville matrix which is not amenable.

\begin{example} \label{ex:amenable-non-SL}
Let $\LL(\bZ^d) = \bZ^d \times \big[\bigoplus_{x \in \bZ^d}\bZ_2 \big]$ where $\bigoplus_{x \in \bZ^d}\bZ_2$ are finitely supported functions on $\bZ^d$. Then, $\LL(\bZ^d)$ can be imbued with the lamplighter group multiplication given by $(x,w) \cdot (y,u) = (x+y, w+ T_x(u))$ where $T_x(u)$ is given by $T_x(u)(z) = u(z-x)$. From \cite[Example 6.1]{Kai83} we know that $\LL(\bZ^d)$ is solveable (and hence amenable) so that by Kesten's amenability criterion, any symmetric random walk on $\LL(\bZ^d)$ is amenable. On the other hand, by \cite[Proposition 6.1]{Kai83} for $d\geq 3$ we know that any symmetric random walk on $\LL(\bZ^d)$ determined by a \emph{finitely supported} measure $\mu$ on $\LL(\bZ^d)$ fails to be \emph{Liouville}. More precisely, this means there is a non-constant \emph{bounded} harmonic function for the random walk. 

Hence, any symmetric random walk on $\LL(\bZ^d)$ arising from a finitely supported measure $\mu$ is amenable but not strong Liouville. Suppose such a random walk is given by a stochastic matrix $P$. By Corollary \ref{c:amenable-strong-liouville-char} there exists another stochastic matrix $Q$ on $\LL(\bZ^d)$ which is Doob equivalent to $P$ (where Doob transform is done via a harmonic function), but $Q$ is not conjugate to $P$. Since $Q$ is a Doob transform of $P$ via a harmonic function and $P$ is amenable, we see that $Q$ is also amenable. Hence we have two non-conjugate amenable stochastic matrices whose tensor algebras are isometrically isomorphic. This shows that recurrence in \cite[Theorem 3.11]{DOM14} cannot be weakened to amenability without additional assumptions.
\end{example}

\begin{example}
Let $\Omega = \bN$ and define
\begin{equation*}
P_{00} = P_{i, i-1} = 0.1, \,\, P_{i-1, i} = 0.9 \,\, \FORAL i \geq 1
, \, \AND P_{ij} = 0 \,\, \text{otherwise}.
\end{equation*}
The spectral radius is strictly less than 1 since
\begin{align*}
    \p{00}{n}
    &\leq  \sum_{i = 0}^{\lceil n/2 \rceil} \binom{n}{i}(0.9)^i(0.1)^{n-i}  \\
    &\leq  (0.9)^{\lceil n/2 \rceil} (0.1)^{\lfloor n/2 \rfloor} 2^n,
\end{align*} 
hence $\rho(P) < 0.6 < 1$ and $P$ is not amenable. 

On the other hand $P$ is strong Liouville. Indeed, let $h$ be a harmonic function for $P$ and let $Q:= P^{(h,1)}$ be its Doob transform. From Lemma  \ref{l:doob-ratio} we see that $1 = \lambda = \frac{P_{ii}^{(n)}}{Q_{ii}^{(n)}}$ for any $n\in \bN$ and $(i,i) \in \Gr(P^n)$. Hence, $P_{ii}^{(n)} = Q_{ii}^{(n)}$ for all $n \in \bN$. Since $Q$ as a stochastic matrix is completely determined by $Q_{00} = P_{00} = 0.1$, and $Q_{ii}^{(2)} = P_{ii}^{(2)} = 0.18$ for $i \in \Omega$, we see that $P=Q$ and $h$ is constant.
\end{example}

We conclude this section with our motivating application to the classification of non-self-adjoint operator algebras arising from stochastic matrices, which follows from Corollary \ref{c:amenable-strong-liouville-char} together with \cite[Theorem 3.8]{DOM14} and \cite[Theorem 7.27]{DOM14}.

\begin{corollary} \label{c:optimal-classification}
Let $P$ and $Q$ be stochastic matrices over $\Omega_P$ and $\Omega_Q$. Suppose that $\rho(P) = \rho(Q)$ and that either $P$ or $Q$ are strong Liouville. Then there is an isometric isomorphism $\varphi : \T_+(P) \rightarrow \T_+(Q)$ if and only if $P$ and $Q$ are conjugate.
\end{corollary}

%%%%%%%%%%%%%%%%%%%%%%%%%%%%%%%%
\section{Non-commutative peaking} \label{S:NC-peaking}
%%%%%%%%%%%%%%%%%%%%%%%%%%%%%%%%

In this section we focus only on \emph{finite} irreducible stochastic matrices. We will characterize the set of boundary states $\Omega_b$ in a way that allows us to detect them by inspecting the stochastic matrix.

When the matrix $P$ is finite, the definition of the algebra $\T_+(P)$ simplifies. For any $(i,j) \in \Gr(P^n)$ and $n \geq 0$ define an operator $S_{ij}^{(n)} \in B(\H_P)$ by setting
$$
S_{ij}^{(n)}(\xi_{m,j',k}) = \begin{cases}
\sqrt{\frac{P_{ij}^{(n)}P_{jk}^{(m)}}{P_{ik}^{(n+m)}}} \xi_{n+m,i,k} & \text{ if } j = j' \\
0 & \text{ otherwise}.
\end{cases}
$$
Then $S_{ij}^{(n)}$ defines a bounded operator on $\H_P$, and its adjoint is given by
$$
S_{ij}^{(n)*}(\xi_{m+n,j',k}) = \begin{cases}
\sqrt{\frac{P_{ij}^{(n)}P_{jk}^{(m)}}{P_{ik}^{(n+m)}}} \xi_{m,j,k} & \text{ if } i = j' \\
0 & \text{ otherwise}.
\end{cases}
$$
Now if $A = [a_{ij}] \in \Arv(P)_n$, since the matrix is now finite we may write $A = \sum_{i,j} a_{ij}S^{(n)}_{ij}$. Hence, we see that 
$$
\T_+(P) := \overline{\Alg} \{ \ S_{ij}^{(n)} \ | \ (i,j) \in \Gr(P^n), \ n \geq 0 \ \}.
$$

\begin{definition}
A state $k\in\Omega$ is said to be \emph{escorted} if for all $s \in\Omega$ with $P_{ks} > 0$, there exists $k\neq k' \in \Omega$ such that  $P_{k's} > 0$.
\end{definition}

It follows that if $k\in \Omega$ is an escorted state for a stochastic matrix $P$, then it is escorted for all iterates of $P$. Indeed, suppose $k\in \Omega$ is escorted. If $P_{ks}^{(n)}>0$, then there exists a path $k, i_1, \dots i_{n-1}, s$ in the sense that $P_{ki_1}, P_{i_1 i_2}, \dots, P_{i_{n-1} s} >0$. Since $k$ is escorted, $\exists \, k' \neq k$ such that $P_{k'i_1} > 0$, and thus $\p{k's}{n} > 0$. Thus, we see that $k$ is an escorted state of $P^n$ for every $n\geq 1$.

By definition, if a state $k$ is escorted, it is not the only element in its cyclic component. Thus $k$ is not exclusive. In addition, if a matrix $P$ has multiple arrival, then all non-exclusive states are escorted. The following strengthens \cite[Proposition 3.11]{DOM16}.

\begin{theorem}\label{t:esc-bdry}
Let $P$ be a finite irreducible matrix and $k\in \Omega$ a state. Then $k\in \Omega_b$ if and only if it is escorted.
\end{theorem}

\begin{proof}
We first show that if a state $k$ is not escorted, then it is not in the boundary. Let $k$ be a non-escorted state, then $\exists\, s\in\Omega$ such that $P_{ks} > 0 $ while $P_{k's} = 0 \FORAL k'\neq k \in \Omega$. By irreducibility of the graph, we must have $s \neq k$. For the sake of brevity, we denote $\H_{P, k}$ as $\H_k$. Define the isometry
\[
    W:\mathcal{H}_{k}\rightarrow \mathcal{H}_{s},\quad \xi_{m,j,k}\mapsto \xi_{m+1,j,s}.
\]
Note that this is well-defined since $\p{js}{m+1} \geq \p{jk}{m}P_{ks} > 0$.
We claim that $T \circ W = W\circ T$ for all operators $T \in\mathcal{T}_+(P)$. It suffices to check this on generators $T = S_{ij}^{(n)}$ applied to basis vectors $\xi_{m,j',k}$ in $\mathcal{H}_{k}$. We assume that $j=j'$, for otherwise both sides of the equation are zero. In this case we have
\[
    (W S_{ij}^{(n)})(\xi_{m,j,k})
    = \sqrt{\frac{P_{ij}^{(n)}P_{jk}^{(m)}}{P_{ik}^{(n+m)}}} W(\xi_{m+n,i,k}) = \sqrt{\frac{P_{ij}^{(n)}P_{jk}^{(m)}}{P_{ik}^{(n+m)}}}
    \xi_{m+n+1,i,s},
\]
and
\[
    (S_{ij}^{(n)} W)(\xi_{m,j,k}) = S_{ij}^{(n)}(\xi_{m+1,j,s})
    = \sqrt{ \frac{P_{ij}^{(n)}P_{js}^{(m+1)}} {P_{is}^{(n+m+1)}}  }\xi_{m+n+1, i, s}.
\]
Note that $\xi_{m+n+1,i,s}$ is well-defined since
\[
    \p{is}{m + n + 1} \geq \p{ij}{n} \p{jk}{m} P_{ks} > 0.
\]
Now as $k$ is not escorted, we see that
\[
    \frac{P_{js}^{(m+1)}}{P_{jk}^{(m)}}
    = \frac{\sum_{k'\in\Omega}P_{jk'}^{(m)}P_{k's}} {P_{jk}^{(m)}}
    = \frac{P_{jk}^{(m)}P_{ks}}{P_{jk}^{(m)}}
    = P_{ks} = \frac{P_{is}^{(m+n+1)}}{P_{ik}^{(n+m)}}.
\]
where the final equality is established through a similar computation. By linearity, the claim holds for all $v\in\H_k$ and generators $S_{ij}^{(n)}$. For arbitrary generators $T_1$ and $T_2$ we have $T_1 T_2 W = T_1 W T_2 = W T_1 T_2$. Thus, the claim is proved for all polynomials in the generators $S_{ij}^{(n)}$, which are dense in $\T_+(P)$. By continuity the claim is proved.

Thus, for any $T = [T_{pq}] \in M_{\ell}(\T_+(P))$ we have
\[
  \|[T_{pq}|_{\H_k}]\| = \| W^{(\ell)} [T_{pq}|_{\H_k}] \| = \| [T_{pq}|_{\H_s}] W^{(\ell)} \| \leq  \|[T_{pq}|_{\H_s}]\|,
\]
where $W^{(\ell)}$ is the $\ell$-fold direct sum of $W$. Hence, we see that $k\notin \Omega_b$ as asserted.

Conversely, suppose $k$ is escorted. Let the cyclic component of $k$ be $\Omega_0$. Since $k$ is escorted, it does not comprise its own cyclic component and $|\Omega_0|\geq 2$. By \cite[Theorem 3.10]{DOM14} there exists $n_0 \in \bN$ such that $P_{kk'}^{(n_0)} > 0 \FORAL k' \in \Omega_0$. For such $n_0$ we claim that $k$ is completely peaking with operator $S_{kk}^{(n_0)}$.

Note that since $\| S_{kk}^{(n_0)}(\xi_{0, k, k})\| = 1$, we must have that $\| S_{kk}^{(n_0)} \| = 1$. Fix $s \neq k$ in $\Omega$ and let $R_{ks} = \{ \ m \ | \ \p{ks}{m} > 0 \ \}$. Using the $C^*$-identity and the formula for the adjoint of $\s{kk}{n_0}$, it suffices to show that,
\[
    ||\s{kk}{n_0}|_{\H_s}||^2 = \sup_{m\in R_{ks}} \frac{\p{kk}{n_0}\p{ks}{m}}{\p{ks}{m+n_0}} < 1.
\]

First assume the supremum is attained at some $m_0 \in \bN$. Since $k$ is escorted, $\exists \, k'\neq k\in\Omega$ such that $\p{k's}{m_0} > 0$. Since $k'\in\Omega_0$, we have that $\p{kk'}{n_0} > 0$, and we get
\[
    \frac{\p{kk}{n_0}\p{ks}{m_0}}{\p{ks}{m_0 + n_0}} \leq 
    \frac{\p{kk}{n_0}\p{ks}{m_0}}{\p{kk'}{n_0}\p{k's}{m_0} + \p{kk}{n_0}\p{ks}{m_0}}
    < 1.
\]

If the supremum is not attained by any finite $m$, then by convergence theorem for finite irreducible matrices (see for instance \cite[Theorem 1.10]{DOM16} for a statement) we get that,  
\[
    ||\s{kk}{n_0}|_{\H_s}||^2 = \limsup_{m\rightarrow\infty} \frac{\p{kk}{n_0}\p{ks}{m}}{\p{ks}{m+n_0}} =
    \p{kk}{n_0} < 1.
\]
Therefore, we conclude that $k\in\Omega_b$ as asserted.
\end{proof}

The following example shows that for matrices without multiple arrival, it is in general not true that a state is either exclusive or in the boundary.

\begin{example} \label{e:nmp}
Let $\Omega=\{1,2\}$. Consider a $2 \times 2$ irreducible stochastic matrix $P$ such that
$$
P = \bbordermatrix{
    & 1 & 2 \cr
    1 & 0 & 1 \cr
    2 & 1/2 & 1/2
}
$$
Both states are non-exclusive, and $P$ does not have multiple arrival. State $1$ is escorted, while state $2$ is not. By Theorem \ref{t:esc-bdry} we see that $\Omega_b = \{1\} \neq \Omega\setminus \Omega_e$. 
\end{example}

We say that a finite irreducible stochastic matrix $P$ is a \emph{cycle}, if $\Gr(P)$ is a simple cycle as a directed graph. We conclude this section by showing that when $P$ is not a cycle, escorted states are exactly those that are necessary in order to retain the norm of any operator in $M_{\ell}(\T_+(P))$. For the proof we will assume some familiarity with the preliminaries on boundary representations of operator algebras presented in \cite[Section 1]{DOM16}, as well as \cite[Section 3]{DOM16} for the boundary representations of $\T_+(P)$.

\begin{corollary} \label{c:max-mod}
Let $P$ be a finite irreducible stochastic matrix which is \emph{not a cycle}, and let $T=[T_{rs}] \in M_{\ell}(\T_+(P))$ for $1\leq \ell \in \bN$. Then $\|T \|$ is the maximum over $\|[T_{pq}|_{\H_{P,k}}]\|$ for escorted states $k\in \Omega$.
\end{corollary}

\begin{proof}
Let $k \in \Omega$. From the proof of $(4) \implies (1)$ in \cite[Corollary 3.17]{DOM16} together with \cite[Theorem 7.2]{Arv11} we see that there exists $k\in \Omega$ which is completely peaking (this coincides with $\pi_k$ being strongly peaking as defined in \cite[Definition 7.1]{Arv11}). Hence, we see that $\Omega_b \neq \emptyset$.

In general, as each $\H_{P,k}$ is reducing, we have that
$\| T \| = \max_{k \in \Omega}\|[T_{pq}|_{\H_{P,k}}]\|$. However, if $k \in \Omega$ is not escorted, by Theorem \ref{t:esc-bdry} there is some $s \in \Omega$ such that $\|[T_{pq}|_{\H_{P,k}}]\| \leq \|[T_{pq}|_{\H_{P,s}}]\|$ for all $T\in \M_{\ell}(\T_+(P))$ and $\ell \geq 1$. Thus, we may inductively remove all non escorted states while retaining the norm of $T$. Eventually we will get that $\|T \|$ is the maximum over $\|[T_{pq}|_{\H_{P,k}}]\|$ for escorted states $k\in \Omega$.
\end{proof}

Theorem \ref{t:esc-bdry} now allows us to concretely describe the $C^*$-envelope $C^*_e(\T_+(P))$ of $\T_+(P)$ as a short exact sequence, as is given at the beginning of \cite[Section 5]{DOM16}, in terms of escorted states. Another useful consequence of Theorem \ref{t:esc-bdry} is a computable form of \cite[Theorem 5.6]{DOM16}, which shows that the column nullity (see \cite[Definition 5.3]{DOM16}) need only be computed for escorted states when trying to determine the $C^*$-envelope $C^*_e(\T_+(P))$ up to *-isomorphism for varying $P$.

\subsection*{Acknowledgments} The authors are grateful to Florin Boca for providing remarks on this manuscript. We are indebted to Florin for his support, help and advice throughout the course of the IGL project. The authors are also grateful to Ariel Yadin for several useful remarks and references.

%%%%%%%%%%%%%%%%%%%%%%%%%%%%%%%%%%%%%%%%%%%

\end{document}